\documentclass{aptpub}
\authornames{TURNER AND WHITEHEAD} % insert the authors here for use in running head
\shorttitle{Partial stochastic dominance} % insert short title here for use in running head

\begin{document}

\title{Partial stochastic dominance for the multivariate Gaussian distribution} % insert title - use \\ if it requires more than one line.

\authorone[Lancaster University]{Amanda Turner} % Affiliation is just the name of your university or institution
\addressone{Department of Mathematics and Statistics, Lancaster University, Lancaster, LA1 4YF, United Kingdom} % Your postal address goes here.
\emailone{a.g.turner@lancaster.ac.uk}
\vspace{-5ex} %%remove this for submission
\authortwo[Lancaster University]{John Whitehead}
\addresstwo{Department of Mathematics and Statistics, Lancaster University, Lancaster, LA1 4YF, United Kingdom} % Your postal address goes here.
\emailtwo{j.whitehead@lancaster.ac.uk}

\begin{abstract}
Gaussian comparison inequalities provide a way of bounding probabilities relating to multivariate Gaussian random vectors in terms of probabilities of random variables with simpler correlation structures. In this paper, we establish the partial stochastic dominance result that the cumulative distribution function of the maximum of a multivariate normal random vector, with positive intraclass correlation coefficient, intersects the cumulative distribution function of a standard normal random variable at most once. This result can be applied to the Bayesian design of a clinical trial in which several experimental treatments are compared to a single control.
\end{abstract}

 % insert keywords separated by a semicolon

 % insert the primary Maths Subject Classification number in the first bracket
         % and the secondary ams number(s) in the second bracket
         % e.g. \ams{60E20}{49G03;49F10}

\section{Introduction}

Gaussian comparison inequalities provide a useful tool in probability and statistics, with applications in areas including Gaussian processes and extreme value theory. A survey of results and applications can be found in the books by Ledoux and Talagrand \cite{lt91} and Lifshits \cite{lif95}. Suppose that $\boldsymbol{X}=(X_1, \dots, X_k)$ and $\boldsymbol{Y}=(Y_1, \dots, Y_k)$ are two multivariate Gaussian vectors. Comparison inequalities typically involve finding conditions on the correlation structures of $\boldsymbol{X}$ and $\boldsymbol{Y}$ from which it can be deduced that $\mathbb{P}(\boldsymbol{X} \in C) \leq \mathbb{P}(\boldsymbol{Y} \in C)$ for some suitable class of sets $C \in \mathbb{R}^k$, usually of the form $\prod_{i=1}^k (-\infty, x_i]$. An important example is Slepian's inequality \cite{slep62} which states that if $\mathbb{E}(\boldsymbol{X})=\mathbb{E}(\boldsymbol{Y})$, $\mathbb{E}(
X_i^2)=\mathbb{E}(Y_i^2)$ for all $i$ and $\mathbb{E}(X_iX_j)\leq\mathbb{E}(Y_iY_j)$ for all $i \neq j$, then $\mathbb{P}(X_1 \leq x_1, \dots, X_k \leq x_k) \leq \mathbb{P}(Y_1 \leq x_1, \dots, Y_k \leq x_k)$ for all $(x_1, \dots, x_k) \in \mathbb{R}^k$.  

A direct consequence of Slepian's inequality is that $F_{X^*}(x) \leq F_{Y^*}(x)$ for all $x \in \mathbb{R}$, where $X^* = \max \{ X_1, \dots, X_k \}$ and $Y^* = \max \{ Y_1, \dots, Y_k \}$, so $X^*$ stochastically dominates $Y^*$ and the distribution functions of $X^*$ and $Y^*$ never cross each other. In this paper, by contrast, we obtain a {\it partial} stochastic dominance result by showing that, under certain assumptions on $\boldsymbol{X}$, $F_{X^*}(x)$ intersects the standard Gaussian distribution function $\Phi(x)$ {\it at most once}. Specifically, we assume that $\boldsymbol{X}$  is a multivariate normal random vector with variances equal to 1 and covariances equal and positive (that is, $\boldsymbol{X}$ has the intraclass correlation structure with positive correlation coefficient). Our partial stochastic dominance result shows that there are three possible cases that can arise, depending on the values of the expectation and correlation coefficient of $\boldsymbol{X}$: the standard Gaussian distribution dominates that of $X^*$; the 
distribution of $X^*$ dominates the standard Gaussian distribution; or there exists some value $x_0 \in \mathbb{R}$ such that $X^*$ dominates the standard Gaussian on the interval $(-\infty, x_0)$ but the standard Gaussian dominates $X^*$ on $(x_0, \infty)$. 

Multivariate normal random vectors with the intraclass correlation structure occur in random effects models in which the error in a measurement arises as a combination of a class-specific error and an individual-specific error. More precisely, $X_{i} = \mu_i + \sqrt{\rho} Y_0 + \sqrt{1-\rho} Y_{i}$ for $i=1, \dots, k,$ where $\rho \in (0,1)$ and the $Y_0, \dots, Y_k$ are independent standard normal random variables. Our motivation for this work was an application to the Bayesian design of exploratory clinical trials in which $k$ experimental treatments are compared to a single control \cite{wct14}. In that paper, one or more of the treatments is suitable to be developed further in a phase III trial if there is a sufficiently high probability that at least one treatment out-performs the control by a given threshold. Corollary \ref{maincor} enables us to quantify the effect of increasing the threshold on that probability. This is then used to recommend an appropriate sample size for the trial. 

The main results are stated in Section \ref{resultsec} and proved in Section \ref{proofsec}. The proof is surprisingly long and technical, as well as being very sensitive to the assumptions. We are not aware of of any simplifications to the argument, however, nor of other results in the literature that enable a comparison of this form. 

\section{Statement of results}
\label{resultsec}

In this section, we state our main theorem, which is then proved in Section \ref{proofsec}. We also state and prove the corollary of this result that is used in \cite{wct14}.

We begin with some notation. For $\rho \in (0,1)$ and $\boldsymbol{\mu} = (\mu_1, \dots, \mu_k) \in \mathbb{R}^k$, let $(X_1, \dots, X_k) \sim N (\boldsymbol{\mu}, \boldsymbol{\Sigma})$ be a multivariate Gaussian random vector with $\Sigma_{ij}=\rho + (1-\rho) \delta_{ij}$, where $\delta_{ij}$ is the Kronecker delta. Let $X^* = \max \{ X_1, \dots, X_k \}$. For any random variable $Y$, we denote the cumulative distribution function of $Y$ by $F_Y$ and the probability density function of $Y$ by $f_Y$. In the special case when $Y \sim N(0,1)$, we set $\Phi=F_Y$ and $\phi=f_Y$.

\begin{theorem}
\label{mainthm}
For any $\rho \in (0,1)$ and $\boldsymbol{\mu} \in \mathbb{R}^k$, the cumulative distribution functions $F_{X^*}(x)$ and $\Phi(x)$ intersect at most once. Furthermore, if $F_{X^*}(x_0) = \Phi(x_0)$ for some $x_0 \in \mathbb{R}$, then $f_{X^*}(x_0) > \phi(x_0)$. 
\end{theorem}

A direct consequence of this result is that 
\begin{align*}
F_{X^*}(x) > \Phi(x) & \quad \mbox{ for all } x>x_0; \\
F_{X^*}(x) < \Phi(x) & \quad \mbox{ for all } x<x_0. 
\end{align*} 
Equivalently, if $Z \sim N(0,1)$, then the conditional distribution of $X^*|(X^*>x_0)$ is stochastically dominated by the conditional distribution of $Z | (Z > x_0)$ and the conditional distribution of $X^*|(X^*<x_0)$ stochastically dominates the conditional distribution of $Z | (Z < x_0)$. 

\begin{remark}
Observe that, for each $i=1, \dots, k$, $F_{X^*}(x) \leq \Phi(x-\mu_i)$ for all $x \in \mathbb{R}$. Therefore, $F_{X^*}(x)$ and $\Phi(x)$ do not intersect if $\max\{\mu_1, \dots, \mu_k\} > 0$. In the degenerate case $\rho=1$ and $\boldsymbol{\mu}=0$, $F_{X^*}(x)=\Phi(x)$ for all $x \in \mathbb{R}$.
\end{remark}

\begin{corollary}
\label{maincor}
Suppose that $\mathbb{P}(X_i < 0 \mbox{ for all } i=1, \dots, k) \geq \kappa$ for some $\kappa \in (0,1)$. Then $\mathbb{P}(X_i < \Phi^{-1}(\zeta) - \Phi^{-1}(\kappa) \mbox{ for all } i=1, \dots, k ) > \zeta$ for all $\zeta \in (\kappa, 1)$.
\end{corollary}
\begin{proof}[Proof of Corollary \ref{maincor}]
We have 
\begin{align*}
\Phi (\Phi^{-1}(\kappa)) &\leq \mathbb{P}(X_i < 0 \mbox{ for all } i=1, \dots, k) \\
&= \mathbb{P}( X_i + \Phi^{-1}(\kappa) <  \Phi^{-1}(\kappa) \mbox{ for all } i=1, \dots, k ) \\
&= F_{Y^*}(\Phi^{-1}(\kappa)),
\end{align*}
where $Y^* = \max\{X_1 + \Phi^{-1}(\kappa), \dots, X_k + \Phi^{-1}(\kappa) \}$.  If $\zeta > \kappa$, then $\Phi^{-1}(\zeta) > \Phi^{-1}(\kappa)$ and so, applying Theorem \ref{mainthm} to $(X_1 + \Phi^{-1}(\kappa), \dots, X_k + \Phi^{-1}(\kappa))$, gives $F_{Y^*}(\Phi^{-1}(\zeta)) > \Phi(\Phi^{-1}(\zeta)) = \zeta$, as required.
\end{proof}

\section{Proof of main result}
\label{proofsec}

In this section we provide a proof of Theorem \ref{mainthm}. We begin by expressing $F_{X^*}(x)$ in terms of independent identically distributed (i.i.d.)~Gaussian random variables. We then show that the theorem holds provided that a quantity expressed in terms of these i.i.d.~random variables can be shown to be strictly positive. This quantity is obtained as the solution to a first order linear differential equation with variable linear coefficient. Positivity follows by showing that the linear coefficient is negative. The coefficient is expressed in terms of standard univariate Gaussian density functions which then enables us to deduce positivity as a consequence of properties of the inverse Mill's ratio.

For $\rho \in (0,1)$, $\nu_0 \in \mathbb{R}$ and $\boldsymbol{\nu}=(\nu_1, \dots, \nu_k) \in \mathbb{R}^k$, let $Y_0, \dots, Y_k$ be independent Gaussian random variables with $Y_0 \sim N(\nu_0, \rho)$ and $Y_i = N(\nu_i, 1-\rho)$ for $i=1, \dots, k$. Let 
\[
G(\nu_0, \boldsymbol{\nu}) = \mathbb{P}(\max \{ Y_1-Y_0, \dots, Y_k-Y_0 \}< 0).
\]
Observe that $(Y_1 - Y_0, \dots, Y_k-Y_0) \sim N(\boldsymbol{\nu} - \nu_0, \boldsymbol{\Sigma})$ where $\boldsymbol{\Sigma}$ is as defined in Section \ref{resultsec} and hence $F_{X^*}(x) = G(x, \boldsymbol{\mu})= G(0, \boldsymbol{\mu}-x)$. 

As $G(\nu_0, \boldsymbol{\nu})$ is strictly increasing in $\nu_0 \in (-\infty, \infty)$ from 0 to 1, there exists a unique function $g:(0,1) \times \mathbb{R}^k \to \mathbb{R}$ such that $G(g(\zeta, \boldsymbol{\nu}), \boldsymbol{\nu})=\zeta$. So $F_{X^*}(x) = \Phi(x)$ for some $x \in \mathbb{R}$ if and only if $g(\zeta, \boldsymbol{\mu}) = \Phi^{-1}(\zeta)$ for some $\zeta \in (0,1)$. In order to show that there is at most one value of $x$ for which $F_{X^*}(x) = \Phi(x)$, it is enough to show that $h(\zeta, \boldsymbol{\nu})=\Phi^{-1}(\zeta)-g(\zeta, \boldsymbol{\nu})$ is strictly increasing in $\zeta$ or equivalently that
\[
z(\zeta, \boldsymbol{\nu}) = \frac{\partial h}{\partial \zeta}(\zeta, \boldsymbol{\nu}) = \left( \phi(\Phi^{-1}(\zeta)) \right )^{-1} - \frac{\partial g}{\partial \zeta}(\zeta, \boldsymbol{\nu}) > 0.
\]
The remainder of the theorem also follows directly from this result by the argument below. Since 
$G(g(\zeta, \boldsymbol{\nu}), \boldsymbol{\nu}) = \zeta$, differentiating with respect to $\zeta$ gives 
\[
\frac{\partial G}{\partial \nu_0}(g(\zeta, \boldsymbol{\nu}), \boldsymbol{\nu}) \frac{\partial g}{ \partial \zeta}(\zeta, \boldsymbol{\nu}) = 1.
\]
Suppose there exists some $x_0 \in \mathbb{R}$ such that $F_{X^*}(x_0) = \Phi(x_0)$. Let $\zeta_0 = \Phi(x_0)$ so $g(\zeta_0, \boldsymbol{\mu}) = x_0 = \Phi^{-1}(\zeta_0)$. Then 
\[
f_{X^*}(x_0) = \frac{\partial G}{\partial \nu_0}(x_0, \boldsymbol{\mu}) = \left ( \frac{\partial g}{ \partial \zeta}(\zeta_0, \boldsymbol{\mu}) \right )^{-1}  > \phi(\Phi^{-1}(\zeta_0)) = \phi(x_0)
\]
as required.

By symmetry, it is sufficient to prove that $z(\zeta, \boldsymbol{\nu})>0$ in the case $\nu_1 \geq \cdots \geq \nu_k$. 
Now $G(\nu_0, \boldsymbol{\nu}) \to \mathbb{P}(X_1 \leq \nu_0) = \Phi(\nu_0 - \nu_1)$ as $\nu_k \leq \cdots \leq \nu_2 \to -\infty$. Since $G$ and its derivatives are equicontinuous in all variables, it follows that $g(\zeta, \boldsymbol{\nu}) \to \nu_1 + \Phi^{-1}(\zeta)$ and $z(\zeta, \boldsymbol{\nu}) \to \frac{\partial \nu_1}{\partial \zeta} = 0$ as $\nu_k \leq \cdots \leq \nu_2 \to -\infty$. We abusively use the notation $f(\boldsymbol{\nu}^i)$ to denote $\lim_{\nu_k \leq \cdots \leq \nu_{i+1} \to -\infty} f(\boldsymbol{\nu})$, so $z(\zeta, \boldsymbol{\nu}^1) = 0$ for all values of $\zeta \in (0,1)$ and $\nu_1 \in \mathbb{R}$.

Recall that $G(g(\zeta, \boldsymbol{\nu}), \boldsymbol{\nu}) = \zeta$. Differentiating both sides with respect to $\nu_i$, $i=1, \dots, k$, gives
\[
\frac{\partial G}{\partial \nu_i}(g(\zeta, \boldsymbol{\nu}), \boldsymbol{\nu}) + \frac{\partial G}{\partial \nu_0}(g(\zeta, \boldsymbol{\nu}), \boldsymbol{\nu}) \frac{\partial g}{ \partial \nu_i}(\zeta, \boldsymbol{\nu}) = 0
\]
and hence
\[ 
\frac{\partial g}{ \partial \nu_i}(\zeta, \boldsymbol{\nu}) = - Q_i(g(\zeta, \boldsymbol{\nu}), \boldsymbol{\nu})
\]
where
\[
Q_i (\nu_0, \boldsymbol{\nu}) = \frac{\frac{\partial G}{\partial \nu_i}}{\frac{\partial G}{\partial \nu_0}} (\nu_0, \boldsymbol{\nu}).
\]
It follows that 
\begin{align*}
\frac{\partial}{\partial \nu_i} \left ( z(\zeta, \boldsymbol{\nu}) - \left( \phi(\Phi^{-1}(\zeta)) \right )^{-1} \right ) 
&= \frac{\partial}{\partial \zeta} \frac{\partial h}{\partial \nu_i} (\zeta, \boldsymbol{\nu})\\
&= \frac{\partial}{\partial \zeta} \left (Q_i(g(\zeta, \boldsymbol{\nu}), \boldsymbol{\nu})\right ) \\
&= \frac{\partial Q_i}{\partial \nu_0} (g(\zeta, \boldsymbol{\nu}), \boldsymbol{\nu}) \frac{\partial g}{\partial \zeta}(\zeta, \boldsymbol{\nu}) \\
&= - \frac{\partial Q_i}{\partial \nu_0} (g(\zeta, \boldsymbol{\nu}), \boldsymbol{\nu}) \left ( z(\zeta, \boldsymbol{\nu}) - \left( \phi(\Phi^{-1}(\zeta)) \right )^{-1} \right ).
\end{align*}
Therefore, for each $i=1, \dots, k$, if $\nu_j$ is fixed for all $j \neq i$, $z(\zeta, \boldsymbol{\nu}) - \left( \phi(\Phi^{-1}(\zeta)) \right )^{-1}$ is the solution to a first order linear differential equation in $\nu_i$ and hence we can evaluate $z(\zeta, \boldsymbol{\nu}^i)$ inductively for $i=2, \dots, k$ by
\[
z(\zeta, \boldsymbol{\nu}^i) - \left( \phi(\Phi^{-1}(\zeta)) \right )^{-1} = \left ( z(\zeta, \boldsymbol{\nu}^{i-1}) - \left( \phi(\Phi^{-1}(\zeta)) \right )^{-1} \right ) \exp \left ( - \int_{-\infty}^{\nu_i} A_i(t) dt \right ),
\]
where $A_i(t)$ is the limit of 
\[
\frac{\partial Q_i}{\partial \nu_0}(g(\zeta, \nu_1, \dots, \nu_{i-1}, t, \nu_{i+1}, \dots, \nu_k), \nu_1, \dots, \nu_{i-1}, t, \nu_{i+1}, \dots, \nu_k) 
\]
as $\nu_k \leq \cdots \leq \nu_{i+1} \to -\infty$.
Hence
\begin{align*}
 z(\zeta, \boldsymbol{\nu}^i) = &\ z(\zeta, \boldsymbol{\nu}^{i-1}) \exp \left ( - \int_{-\infty}^{\nu_i} A_i(t) dt \right ) \\
&\ + \left( \phi(\Phi^{-1}(\zeta)) \right )^{-1} \left ( 1 - \exp \left ( - \int_{-\infty}^{\nu_i} A_i(t) dt \right ) \right ).
\end{align*}
Suppose that it is true that $z(\zeta, \boldsymbol{\nu}^{i-1})\geq 0$, for some $i \geq 2$.  Then, if
\begin{equation}
\label{qineq}
\frac{\partial Q_i}{\partial \nu_0}(\nu_0, \boldsymbol{\nu}^i) \geq 0 \mbox{ for all } \nu_0 \in \mathbb{R} \mbox{ and } \nu_1 \geq \cdots \geq \nu_i, 
\end{equation}
with strict inequality unless $\nu_1=\nu_2=\cdots=\nu_i$, then it follows that $z(\zeta, \boldsymbol{\nu}^{i}) > 0$, and by induction the theorem will be proven.

The remainder of this proof is concerned with showing that \eqref{qineq} holds for $i=1, \dots, k$. Since
\[
\frac{\partial Q_i}{\partial \nu_0} = \frac{\partial}{\partial \nu_0} \left ( \frac{\frac{\partial G}{\partial \nu_i}}{\frac{\partial G}{\partial \nu_0}} \right )
= \frac{\frac{\partial^2G}{\partial \nu_0 \partial \nu_i}\frac{\partial G}{\partial \nu_0}- \frac{\partial^2G}{\partial \nu_0^2}\frac{\partial G}{\partial \nu_i}}{\left ( \frac{\partial G}{\partial \nu_0}\right )^2},
\]
it is sufficient to show that
\[
\Delta_i = \left ( \frac{\partial^2G}{\partial \nu_0 \partial \nu_i}\frac{\partial G}{\partial \nu_0}- \frac{\partial^2G}{\partial \nu_0^2}\frac{\partial G}{\partial \nu_i} \right ) (\nu_0, \boldsymbol{\nu}^i) \geq 0
\] 
for all $i=1, \dots, k$, $\nu_0 \in \mathbb{R}$ and $\nu_1 \geq \cdots \geq \nu_i$ with strict inequality unless $\nu_1=\nu_2=\cdots=\nu_i$.

Now
\begin{align*}
G(\nu_0, \boldsymbol{\nu}) &= \mathbb{P}(\max \{ Y_1, \dots, Y_k \} < Y_0) \\
&= \int_{-\infty}^{\infty} f_{Y_0}(t) \mathbb{P}(\max \{ Y_1, \dots, Y_k \} < t) dt \\
&= \frac{1}{\sqrt{\rho}} \int_{-\infty}^{\infty} \phi \left ( \frac{t-\nu_0}{\sqrt{\rho}} \right ) \prod_{j=1}^k \Phi \left ( \frac{t-\nu_j}{\sqrt{1-\rho}} \right ) dt.
\end{align*}
Hence
\begin{align*}
\frac{\partial G}{\partial \nu_0}(\nu_0, \boldsymbol{\nu}) 
= & \ -\frac{1}{\rho} \int_{-\infty}^{\infty} \phi' \left ( \frac{t-\nu_0}{\sqrt{\rho}} \right ) \prod_{j=1}^k \Phi \left ( \frac{t-\nu_j}{\sqrt{1-\rho}} \right ) dt \\
= & \ - \frac{1}{\sqrt{\rho}}\left [ \phi \left ( \frac{t-\nu_0}{\sqrt{\rho}} \right ) \prod_{j=1}^k \Phi \left ( \frac{t-\nu_j}{\sqrt{1-\rho}} \right ) \right ]_{-\infty}^{\infty} \\
  & \ + \frac{1}{\sqrt{\rho}}\int_{-\infty}^{\infty} \phi \left ( \frac{t-\nu_0}{\sqrt{\rho}} \right ) \frac{\partial}{\partial t} \left ( \prod_{j=1}^k \Phi \left ( \frac{t-\nu_j}{\sqrt{1-\rho}} \right ) \right ) dt \\
= \ &	\frac{1}{\sqrt{\rho}}\int_{-\infty}^{\infty} \phi \left ( \frac{t-\nu_0}{\sqrt{\rho}} \right ) \frac{\partial}{\partial t} \left ( \prod_{j=1}^k \Phi \left ( \frac{t-\nu_j}{\sqrt{1-\rho}} \right ) \right ) dt, 
\end{align*}
\[
\frac{\partial G}{\partial \nu_i}(\nu_0, \boldsymbol{\nu})
= - \frac{1}{\sqrt{\rho(1-\rho)}} \int_{-\infty}^{\infty} \phi \left ( \frac{t-\nu_0}{\sqrt{\rho}} \right ) \phi \left ( \frac{t-\nu_i}{\sqrt{1-\rho}} \right ) \prod_{j\neq i} \Phi \left ( \frac{t-\nu_j}{\sqrt{1-\rho}} \right ) dt,
\]
\[
\frac{\partial^2 G}{\partial \nu_0 \partial \nu_i}(\nu_0, \boldsymbol{\nu})
= \frac{1}{\sqrt{\rho(1-\rho)}} \int_{-\infty}^{\infty} \phi' \left ( \frac{t-\nu_0}{\sqrt{\rho}} \right ) \phi \left ( \frac{t-\nu_i}{\sqrt{1-\rho}} \right ) \prod_{j\neq i} \Phi \left ( \frac{t-\nu_j}{\sqrt{1-\rho}} \right ) dt, 
\]
and
\[
\frac{\partial^2 G}{\partial \nu_0^2}(\nu_0, \boldsymbol{\nu})
= \frac{1}{\rho}\int_{-\infty}^{\infty} \phi' \left ( \frac{t-\nu_0}{\sqrt{\rho}} \right ) \frac{\partial}{\partial t} \left ( \prod_{j=1}^k \Phi \left ( \frac{t-\nu_j}{\sqrt{1-\rho}} \right ) \right ) dt. 
\]
It follows that
\begin{align*}
\Delta_i = & \ \frac{1}{\sqrt{\rho(1-\rho)}} \int_{-\infty}^{\infty} \phi' \left ( \frac{s-\nu_0}{\sqrt{\rho}} \right ) \phi \left ( \frac{s-\nu_i}{\sqrt{1-\rho}} \right ) \prod_{j\neq i} \Phi \left ( \frac{s-\nu_j}{\sqrt{1-\rho}} \right ) ds \\
& \ \times \frac{1}{\sqrt{\rho}}\int_{-\infty}^{\infty} \phi \left ( \frac{t-\nu_0}{\sqrt{\rho}} \right ) \frac{\partial}{\partial t} \left ( \prod_{j=1}^k \Phi \left ( \frac{t-\nu_j}{\sqrt{1-\rho}} \right ) \right ) dt \\
& \ - \frac{1}{\sqrt{\rho(1-\rho)}} \int_{-\infty}^{\infty} \phi \left ( \frac{s-\nu_0}{\sqrt{\rho}} \right ) \phi \left ( \frac{s-\nu_i}{\sqrt{1-\rho}} \right ) \prod_{j\neq i} \Phi \left ( \frac{s-\nu_j}{\sqrt{1-\rho}} \right ) ds \\
& \ \times \frac{1}{\rho}\int_{-\infty}^{\infty} \phi' \left ( \frac{t-\nu_0}{\sqrt{\rho}} \right ) \frac{\partial}{\partial t} \left ( \prod_{j=1}^k \Phi \left ( \frac{t-\nu_j}{\sqrt{1-\rho}} \right ) \right ) dt \\
= & \ \frac{1}{\rho^2 \sqrt{1-\rho}} \int_{-\infty}^{\infty} \int_{-\infty}^{\infty} \Bigg \{ (s-t) \phi \left ( \frac{s-\nu_0}{\sqrt{\rho}} \right ) \phi \left ( \frac{t-\nu_0}{\sqrt{\rho}} \right ) \phi \left ( \frac{s-\nu_i}{\sqrt{1-\rho}} \right )  \\
& \ \times  \prod_{j\neq i} \Phi \left ( \frac{s-\nu_j}{\sqrt{1-\rho}} \right ) \frac{\partial}{\partial t} \left ( \prod_{j=1}^k \Phi \left ( \frac{t-\nu_j}{\sqrt{1-\rho}} \right ) \right ) \Bigg \} ds dt.
\end{align*}
Interchanging $s$ and $t$, we also have
\begin{align*}
\Delta_i = & \ \frac{1}{\rho^2 \sqrt{1-\rho}} \int_{-\infty}^{\infty} \int_{-\infty}^{\infty} \Bigg \{ (t-s) \phi \left ( \frac{s-\nu_0}{\sqrt{\rho}} \right ) \phi \left ( \frac{t-\nu_0}{\sqrt{\rho}} \right ) \phi \left ( \frac{t-\nu_i}{\sqrt{1-\rho}} \right )  \\
& \ \times  \prod_{j\neq i} \Phi \left ( \frac{t-\nu_j}{\sqrt{1-\rho}} \right ) \frac{\partial}{\partial s} \left ( \prod_{j=1}^k \Phi \left ( \frac{s-\nu_j}{\sqrt{1-\rho}} \right ) \right ) \Bigg \} ds dt.
\end{align*}
Hence, $\Delta_i$ can be expressed as the average of these two forms, giving
\begin{align*}
\Delta_i = & \ \frac{1}{2\rho^2 \sqrt{1-\rho}} \int_{-\infty}^{\infty} \int_{-\infty}^{\infty} \Bigg \{ \phi \left ( \frac{s-\nu_0}{\sqrt{\rho}} \right ) \phi \left ( \frac{t-\nu_0}{\sqrt{\rho}} \right ) \phi \left ( \frac{s-\nu_i}{\sqrt{1-\rho}} \right ) \phi \left ( \frac{t-\nu_i}{\sqrt{1-\rho}} \right ) \\
& \ \times  \prod_{j\neq i} \Phi \left ( \frac{s-\nu_j}{\sqrt{1-\rho}} \right ) \prod_{j\neq i} \Phi \left ( \frac{t-\nu_j}{\sqrt{1-\rho}} \right ) (s-t) \left ( H_i(s, \boldsymbol{\nu}) - H_i(t, \boldsymbol{\nu})\right )  \Bigg \} ds dt,
\end{align*}
where
\[
H_i(t, \boldsymbol{\nu}) = \frac{\Phi \left ( \frac{t-\nu_i}{\sqrt{1-\rho}} \right )}{\phi \left ( \frac{t-\nu_i}{\sqrt{1-\rho}} \right )}\frac{\frac{\partial}{\partial t} \left ( \prod_{j=1}^k \Phi \left ( \frac{t-\nu_j}{\sqrt{1-\rho}} \right ) \right )}{\prod_{j=1}^k \Phi \left ( \frac{t-\nu_j}{\sqrt{1-\rho}} \right )}.
\]
Now 
\[
(s-t) \left ( H_i(s, \boldsymbol{\nu}) - H_i(t, \boldsymbol{\nu}) \right ) = (s-t)\int_t^s \frac{\partial H_i}{\partial u}(u, \boldsymbol{\nu}) du
\]
which is positive for all $s$ and $t$ if the integrand is positive for all $u$. In order to show that $\Delta_i \geq 0$, it is therefore sufficient show that $\frac{\partial H_i}{\partial t}(t, \boldsymbol{\nu}^i) \geq 0$
for all $i=1, \dots, k$, and $\nu_1 \geq \cdots \geq \nu_i$ with strict inequality unless $\nu_1=\nu_2=\cdots=\nu_i$.

Let $m(x) = \phi(x) / \Phi (x)$ be the inverse Mill's ratio. Then
\[
H_i(t, \boldsymbol{\nu}) = \frac{1}{\sqrt{1-\rho}} \sum_{j=1}^k \frac{m\left ( \frac{t-\nu_j}{\sqrt{1-\rho}} \right )}{m\left ( \frac{t-\nu_i}{\sqrt{1-\rho}} \right )}.
\]
Using the fact that $m'(x)=-m(x)(x+m(x))$, we obtain
\begin{align*}
\frac{\partial H_i}{\partial t}(t, \boldsymbol{\nu}^i) 
& = \frac{1}{1-\rho} \sum_{j=1}^k \left ( \frac{m'\left ( \frac{t-\nu_j}{\sqrt{1-\rho}} \right )}{m\left ( \frac{t-\nu_i}{\sqrt{1-\rho}} \right )} - \frac{m\left ( \frac{t-\nu_j}{\sqrt{1-\rho}} \right )m'\left ( \frac{t-\nu_i}{\sqrt{1-\rho}} \right )}{m\left ( \frac{t-\nu_j}{\sqrt{1-\rho}} \right )^2} \right ) \\
&= \frac{1}{1-\rho} \sum_{j=1}^k \frac{m\left ( \frac{t-\nu_j}{\sqrt{1-\rho}} \right )}{m\left ( \frac{t-\nu_i}{\sqrt{1-\rho}} \right )}\left ( \frac{\nu_j -\nu_i}{\sqrt{1-\rho}} + m\left ( \frac{t-\nu_i}{\sqrt{1-\rho}} \right ) - m\left ( \frac{t-\nu_j}{\sqrt{1-\rho}} \right )\right ) \\
&= \frac{1}{(1-\rho)^{3/2}} \sum_{j=1}^k \frac{m\left ( \frac{t-\nu_j}{\sqrt{1-\rho}} \right )}{m\left ( \frac{t-\nu_i}{\sqrt{1-\rho}} \right )} \int_{\nu_i}^{\nu_j} \left ( 1 + m' \left ( \frac{t-u}{\sqrt{1-\rho}} \right ) \right ) du.
\end{align*}
But $m'(x)>-1$ for all $x \in \mathbb{R}$ (see Sampford \cite{samp53}) and hence, letting $\nu_{k} \leq \dots \leq \nu_{i+1} \to -\infty$,
\[
\frac{\partial H_i}{\partial t}(t, \boldsymbol{\nu}^i) \geq 0
\] 
for all $i=1, \dots, k$, and $\nu_1 \geq \cdots \geq \nu_i$ with strict inequality unless $\nu_1=\nu_2=\cdots=\nu_i$, as required. 

%\acks
% Place the text of your acknowledgements after the \acks command.
% \acks generates the heading "Acknowledgements".
% If you wish to make only one acknowledgement, use \ack.
% \ack generates the heading "Acknowledgement".

\bibliography{multivarbib}

\end{document}